\documentclass[11pt]{article}

\usepackage[latin1]{inputenc}
\usepackage[english]{babel}
\usepackage{here,enumerate,graphicx}
\usepackage{amsmath,amssymb,amscd,amsthm,mathrsfs}
\usepackage[colorlinks=true,linkcolor=blue]{hyperref}

\usepackage[flushmargin]{footmisc}
\newcommand\blfootnote[1]{%
	\begingroup
	\renewcommand\thefootnote{}\footnote{#1}%
	\addtocounter{footnote}{-1}%
	\endgroup}

\theoremstyle{plain}
\newtheorem{theorem}{Theorem}[section]

\newtheorem{corollary}[theorem]{Corollary}

\theoremstyle{definition}

\newtheorem{remark}[theorem]{Remark}
\newtheorem{example}[theorem]{Example}
\newtheorem{question}[theorem]{Question}
\newtheorem{problem}[theorem]{Problem}

\newcommand{\NN}{\mathbb{N}}
\newcommand{\ZZ}{\mathbb{Z}}

\newcommand{\CC}{\mathbb{C}}
\newcommand{\DD}{\mathbb{D}}

\newcommand{\Fc}{\mathcal{F}}
\newcommand{\Lc}{\mathcal{L}}

\newcommand{\cl}{\overline}
\newcommand{\dist}{\operatorname{dist}}
\newcommand{\lspan}{\operatorname{span}}

\newcommand{\supp}{\operatorname{supp}}

\newcommand{\zero}{\mathfrak{o}}
\renewcommand{\hat}{\widehat}
\newcommand{\Chi}{\text{Chi}}
\newcommand{\dx}{\partial}

\newcommand{\sbf}{\boldsymbol}

\begin{document}
\begin{center}
	\begin{LARGE}
		{\bf Frequently hypercyclic composition operators\\ on the little Lipschitz space of a rooted tree}
	\end{LARGE}
\end{center}

\vspace*{-0.2in}

\begin{center}
	\begin{Large}
		Antoni L\'opez-Mart\'inez\blfootnote{\textbf{Journal-ref}: Mediterranean Journal of Mathematics, Volume 23, article number 87, (2026).\\ \textbf{DOI}: https://doi.org/10.1007/s00009-026-03075-6}
	\end{Large}
\end{center}

\vspace*{-0.2in}

\begin{abstract}
	We characterize the strictly increasing symbols $\varphi:\mathbb{N}_0\longrightarrow\mathbb{N}_0$ whose composition operators~$C_{\varphi}$ satisfy the Frequent Hypercyclicity Criterion on the little Lipschitz space $\mathcal{L}_0(\mathbb{N}_0)$. With this result we continue the recent research about this kind of spaces and operators, but our approach relies on establishing a natural isomorphism between the Lipschitz-type spaces over rooted trees and the classical spaces $\ell^{\infty}$ and $c_0$. Such isomorphism provides an alternative framework that simplifies and allows to improve many previous results about these spaces and the operators defined there.
\end{abstract}

\vspace*{-0.2in}

\section{Introduction}

Within the domains of Functional Analysis and Operator Theory, the field of Linear Dynamics has established itself over the past five decades as a highly active and influential area of research, with a primary focus on the notion of {\em hypercyclicity} (i.e.\ existence of a dense orbit). The first examples of hypercyclic operators acting on Banach spaces were the {\em Rolewicz operators} \cite{Rolewicz1969}. These are the multiples $\lambda B$ for any $|\lambda|>1$ of the backward shift $B:X\longrightarrow X$. We recall that $B$ acts on the classical sequence spaces $X=\ell^p(I)$ or $c_0(I)$, indexed by the set of positive integers $I=\NN$ or equivalently by the set of non-negative integers $\NN_0:=\NN\cup\{0\}$, as follows: for each $x=(x_j)_{j\in I} \in X$ one has
\[
x = (x_j)_{j\in I} \longmapsto B(x) := (x_{j+1})_{j\in I}.
\]
Since then, the more general {\em weighted shifts} on sequence spaces indexed by $\NN$ or $\NN_0$ have become one of the main sources of examples in Linear Dynamics (see \cite[Chapter~4]{GrPe2011_book} and the references therein).

These classical weighted shifts were later generalized into the so-called {\em weighted shifts on trees}, for which the underlying spaces are again sequence Banach (or even Fr\'echet) spaces but this time indexed by what is called a {\em directed tree} $T$ (see Section~\ref{Sec_2:preliminar} for a precise definition). In fact, and among others, references such as \cite{AA2024,GrPa2023,GrPa2023_arXiv,JabJungStoc2012_book,LopezPa2025_JDE_shifts,Martinez2017} have focused on this topic by letting the underlying spaces to be the respective $\ell^p(T)$ or $c_0(T)$, being these natural generalizations of $\ell^p(\NN_0)$ and $c_0(\NN_0)$.

Recently, two apparently new sequence Banach spaces indexed by a tree~$T$ have appeared in the literature: the {\em Lipschitz space} $\Lc(T)$ and the {\em little Lipschitz space} $\Lc_0(T)$. These Banach spaces were introduced in the 2010 paper \cite{CoEas2010} and, since then, many works have been published studying how shift, multiplication and composition operators behave there (see \cite{ACoEas2011,ACoEas2012,ACoEas2013,ACoEas2014,CoEas2012,CoMar2017,ColonnaMar2025_MJM_composition,MarRi2020}). The main objective of this paper is to provide examples of frequently hypercyclic composition operators on the little Lipschitz space of $\NN_0$, addressing the following open question left in  \cite[Subsection~6.1]{ColonnaMar2025_MJM_composition}:

\begin{question}\label{Q:first}
	Are there some  both sufficient and necessary conditions for the hypercyclicity of a composition operator $C_{\varphi}$ acting on a little Lipschitz space $\Lc_0(T)$ of a rooted tree $T$?
\end{question}

In this paper we solve Question~\ref{Q:first} for the particular case of composition operators whose symbols are strictly increasing self-maps of the set of non-negative integers. Our main result shows that these operators are hypercyclic if and only if they satisfy the Frequent Hypercyclicity Criterion.

The paper is organized as follows. In Section~\ref{Sec_2:preliminar} we introduce the notation used and we establish an isometric isomorphism between the Lipschitz-type spaces $\Lc(T)$ and $\Lc_0(T)$, of each rooted tree~$T$, and the spaces $\ell^{\infty}(T)$ and $c_0(T)$. This isomorphism allows us to compute the norms of the composition, multiplication and shift operators on $\Lc(T)$ and $\Lc_0(T)$, improving all the norm estimates given in the literature for these operators (see Subsections~\ref{SubSec_2.2:composition} and \ref{SubSec_2.3:mult_shift}). In Section~\ref{Sec_3:existence} we again use the established isomorphism, together with the Frequent Hypercyclicity Criterion, in order to characterize those strictly increasing symbols $\varphi:\NN_0\longrightarrow\NN_0$ for which the composition operators~$C_{\varphi}$ are hypercyclic on $\Lc_0(\NN_0)$. We conclude the paper in Section~\ref{Sec_4:conclusions}, where we briefly discuss further applications of the isomorphism established above and formulate several open problems.

\section{Notation, a useful isomorphism, and some consequences}\label{Sec_2:preliminar}


Following \cite{ColonnaMar2025_MJM_composition}, a {\em directed tree} (or just a {\em tree}) will be a locally finite, connected, and simply-connected graph $T$ that as a set we identify with the collection of its vertices. We will say that two distinct vertices $v\neq w$ are {\em adjacent}, and we will denote it by $v \sim w$, if $v$ and $w$ are connected by an edge of the graph $T$. The {\em degree} of a vertex $v \in T$ is the number of vertices in $T$ that are adjacent to $v$, and a~{\em path} between two vertices $v\neq w$ will be a finite sequence of distinct vertices $[v,v_1,v_2,...,v_{k-1},w]$~in the tree~$T$ and fulfilling that $v \sim v_1 \sim v_2 \sim \cdots \sim v_{k-1} \sim w$ (note that in a connected and simply-connected graph there always exists a unique path between two distinct vertices).

Also as in \cite{ColonnaMar2025_MJM_composition}, every tree $T$ considered in this paper will be a {\em rooted tree}, i.e.\ there will be a distinguished vertex $\zero \in T$ that will be called the {\em root} of~$T$. Given two vertices $v\neq w$ in a rooted tree~$T$ we will say that $v$ is a {\em descendant} of $w$, or that $w$ is an {\em ancestor} of $v$, if the vertex $w$ lies in the unique path between~$\zero$ and~$v$. We will write $T^* := T\setminus\{\zero\}$ and given any vertex $v \in T^*$ we will denote by~$v^-$ the unique vertex in $T$ that is ``an ancestor of'' and ``adjacent to'' $v$. In this case $v^-$ is called the {\em parent} of~$v$, and~$v$ is said to be a {\em child} of~$v^-$. In addition, given any vertex $v \in T$ we will denote by $\Chi(v)$ the set formed by all the children of $v$ (note that $v \in T$ could be the root $\zero$ this time).

Given $v,w \in T$ we will denote by $\dist(v,w)$ the number of edges in the unique path between~$v$~and~$w$, calling it the {\em distance} between $v$ and $w$, and setting $\dist(v,w)=0$ whenever $v=w$. Moreover, for each $v \in T$ we will denote by $|v|:=\dist(\zero,v)$ the distance between the root $\zero$ and the vertex $v$. All the trees considered in this paper will be rooted trees for which only the root $\zero$ may have degree equal to one (so that all these trees will be infinite). Finally, also as in \cite{ColonnaMar2025_MJM_composition}, for each $v \in T$ we will denote by $S_v$ the {\em sector of $v$}, which is the set consisting of $v$ and all its descendants (note that $S_{\zero} = T$).

A {\em function}~$f$ on a tree~$T$ will be a complex-valued function defined on the set of its vertices. The, the {\em Lipschitz} and the {\em little Lipschitz spaces} of a rooted tree $T$ are defined as
\[
\Lc(T) := \left\{ f:T\longrightarrow\CC \ ; \ \sup_{v \in T^*} |f(v)-f(v^-)|<\infty \right\},
\]
\[
\Lc_0(T) := \left\{ f:T\longrightarrow\CC \ ; \ \lim_{|v|\to\infty} |f(v)-f(v^-)|=0 \right\}.
\]
It was proved in \cite{CoEas2010} that $\Lc(T)$ is a Banach space with the norm
\begin{equation}\label{eq:norm_+}
	\|f\| := |f(\zero)| + \sup_{v \in T^*} |f(v)-f(v^-)| \quad \text{ for each } f \in \Lc(T).
\end{equation}
It was also proved in \cite{CoEas2010} that $\Lc_0(T)$ is always a closed and separable subspace of $\Lc(T)$. However, following~\cite{ColonnaMar2025_MJM_composition,MarRi2020}, we will use in both spaces $\Lc(T)$ and $\Lc_0(T)$ the equivalent norm
\begin{equation}\label{eq:norm_L}
	\|f\|_{\Lc} := \max\left\{ |f(\zero)| , \sup_{v \in T^*} |f(v)-f(v^-)| \right\} \quad \text{ for each } f \in \Lc(T).
\end{equation}
The norm $\|\cdot\|_{\Lc}$ allows us to find an interesting \textbf{isometric isomorphism}, which will be our main tool in this paper (see Subsections~\ref{SubSec_2.2:composition}~and~\ref{SubSec_2.3:mult_shift} and Section~\ref{Sec_3:existence}):

\begin{remark}[\textbf{A useful isomorphism}]\label{Rem:1-isom}
	Given a rooted tree $T$, and following \cite{GrPa2023,GrPa2023_arXiv,JabJungStoc2012_book,LopezPa2025_JDE_shifts}, consider the well-known sequence spaces
	\[
	\ell^{\infty}(T) := \left\{ x=(x_v)_{v\in T} \in \CC^{T} \ ; \ \sup_{v \in T} |x_v| < \infty \right\}
	\]
	and
	\[
	c_0(T) := \left\{ x=(x_v)_{v\in T} \in \CC^{T} \ ; \ \lim_{|v|\to\infty} |x_v| = 0 \right\},
	\]
	both of them endowed with the {\em supremum norm}
	\[
	\|x\|_{\infty} := \sup_{v \in T} |x_v| \quad \text{ for each } x=(x_v)_{v\in T} \in \ell^{\infty}(T).
	\]
	The spaces $(\ell^{\infty}(T),\|\cdot\|_{\infty})$ and $(c_0(T),\|\cdot\|_{\infty})$ are sequence Banach spaces, and it is not hard to check that the map $D : (\Lc(T),\|\cdot\|_{\Lc}) \longrightarrow (\ell^{\infty}(T),\|\cdot\|_{\infty})$, defined for each $f \in \Lc(T)$ as
	\[
	f \longmapsto D(f) = ([D(f)]_v)_{v \in V} \in \ell^{\infty}(T) \quad \text{ with } \quad [D(f)]_v :=
	\begin{cases}
		f(\zero), & \text{ if } v=\zero, \\[5pt]
		f(v)-f(v^-), & \text{ if } v \in T^*,
	\end{cases}
	\]
	is indeed an isometric isomorphism. In fact, the inverse operator $D^{-1} : (\ell^{\infty}(T),\|\cdot\|_{\infty}) \longrightarrow (\Lc(T),\|\cdot\|_{\Lc})$ can be easily computed and expressed for each $x=(x_v)_{v \in T} \in \ell^{\infty}(T)$ as
	\[
	x \longmapsto D^{-1}(x) \in \Lc(T) \quad \text{ with } \quad [D^{-1}(x)](v) = \left( \sum_{w \in [\zero,v_1,v_2,...,v]} x_w \right) \quad \text{ for every } v \in T,
	\]
	where $[\zero,v_1,v_2,...,v]$ is the unique path in $T$ between the root $\zero$ and the vertex $v$. Since we will focus on the little Lipschitz space we will use the restriction $D : (\Lc_0(T),\|\cdot\|_{\Lc})\longrightarrow(c_0(T),\|\cdot\|_{\infty})$, which by abuse of notation will still be denoted by $D$. Note that $D$ is also an isometric isomorphism between $(\Lc_0(T),\|\cdot\|_{\Lc})$ and $(c_0(T),\|\cdot\|_{\infty})$, and that the expression of the respective inverse map $D^{-1}:(c_0(T),\|\cdot\|_{\infty}) \longrightarrow (\Lc_0(T),\|\cdot\|_{\Lc})$ coincides with that of $D^{-1} : (\ell^{\infty}(T),\|\cdot\|_{\infty}) \longrightarrow (\Lc(T),\|\cdot\|_{\Lc})$.
\end{remark}

Considering the mathematical objects introduced in Remark~\ref{Rem:1-isom}, let us establish some further notation that will help us to clarify how $D$ and $D^{-1}$ act on $\Lc(T)$ and $\Lc_0(T)$:
\begin{enumerate}[--]
	\item For each $w \in T$ we will denote by $e^{(w)}:=(\delta_{w,v})_{v\in T} \in c_0(T)$ the $w$-th element of the canonical Schauder basis of $c_0(T)$, i.e.\ $\delta_{w,v}=1$ when $w=v$ and $\delta_{w,v}=0$ otherwise. We thus have that
	\[
	D(\chi_{S_w})=e^{(w)} \quad \text{ or equivalently } \quad D^{-1}(e^{(w)})=\chi_{S_w},
	\]
	where $S_w$ is the sector of the vertex $w$, and $\chi_{S_w}:T\longrightarrow\CC$ is the corresponding indicator function.
		
	\item For each $w \in T$ we will use the notation $\dx^{(w)} := e^{(w)} - \sum_{u\in\Chi(w)} e^{(u)} \in c_0(T)$. We thus have that
	\[
	D(\chi_{\{w\}})=\dx^{(w)} \quad \text{ or equivalently } \quad D^{-1}(\dx^{(w)})=\chi_{\{w\}},
	\]
	where $\chi_{\{w\}}:T\longrightarrow\CC$ is the corresponding indicator function.
\end{enumerate}
Moreover, the following linear subspaces of $c_0(T)$, defined as
\[
c_{00}(T) := \lspan\{ e^{(w)} \ ; \ w \in T \}
\]
and
\[
\Sigma_{00}(T) := \lspan\{ \dx^{(w)} \ ; \ w \in T \},
\]
will play a fundamental role in our main result (see Theorem~\ref{The:NN_0} below). Actually, if for each vector $x=(x_v)_{v\in T} \in \ell^{\infty}(T)$ we denote its {\em support} by $\supp(x):=\{ v \in T \ ; \ x_v\neq 0 \}$ and we define
\[
|\supp(x)| :=
\begin{cases}
	\sup\left\{ |v| \ ; \ v \in \supp(x) \right\}, & \text{ if } \supp(x)\neq \emptyset, \\[5pt]
	0, & \text{ if } \supp(x)=\emptyset,
\end{cases}
\]
then it is not hard to check the equalities
\[
c_{00}(T) = \left\{ x=(x_v)_{v\in T} \in c_0(T) \ ; \ |\supp(x)|<\infty \right\}
\]
and
\[
\Sigma_{00}(T) = \left\{ x=(x_v)_{v\in T} \in c_{00}(T) \ ; \ \textstyle \sum_{w \in [\zero,v_1,v_2,...,v]} x_w = 0 \ \text{ whenever } |v|=|\supp(x)| \right\},
\]
where $[\zero,v_1,v_2,...,v]$ is again the unique path in $T$ between the root $\zero$ and the vertex $v$. These last alternative expressions for $c_{00}(T)$ and $\Sigma_{00}(T)$ allow us to easily check that they both are dense linear subspaces of $c_0(T)$. As a consequence, having in mind the isomorphism established in Remark~\ref{Rem:1-isom} together with the direct computations $D^{-1}(e^{(w)})=\chi_{S_w}$ and $D^{-1}(\dx^{(w)})=\chi_{\{w\}}$ for each $w \in T$, one can deduce the results \cite[Theorem~2.3]{CoEas2010} and \cite[Lemma~2.3]{ColonnaMar2025_MJM_composition} in a direct way.

Other results from the previous works on these Lipschitz spaces \cite{ACoEas2011,ACoEas2012,ACoEas2013,ACoEas2014,CoEas2010,CoEas2012,CoMar2017,ColonnaMar2025_MJM_composition,MarRi2020} can also be directly obtained via the isometric isomorphism established in Remark~\ref{Rem:1-isom}. In fact, the findings about their Banach-space structure, such as \cite[Proposition~3.1]{CoMar2017} where it is proved that $\Lc(T)$ is not separable, or \cite[Proposition~4.2]{CoMar2017} where it is proved that the dual of $\Lc_0(T)$ and the pre-dual of $\Lc(T)$ is an $\ell^1$-type space, are now obvious. Also, the results about certain inequalities satisfied by every function $f \in \Lc(T)$, such as \cite[Lemma~3.4]{CoEas2010} or \cite[Lemma~2.2]{MarRi2020}, are indeed direct if we note that $f(v)$ is a sum of $|v|+1$ coordinates of a vector in $\ell^{\infty}(T)$, namely $D(f) \in \ell^{\infty}(T)$. In the aforementioned works, the map~$D$ was introduced as a tool to compute the $\|\cdot\|_{\Lc}$-norm of elements in $\Lc(T)$ and $\Lc_0(T)$, but none of these references addressed the natural isometric isomorphism, nor the explicit form of the inverse operator $D^{-1}$ as presented in Remark~\ref{Rem:1-isom}. Using both $D$ and $D^{-1}$ we can obtain alternative expressions for the operators defined in $\Lc(T)$ and $\Lc_0(T)$. See Section~\ref{Sec_4:conclusions} for more comments.

\subsection{Composition operators}\label{SubSec_2.2:composition}

Given a rooted tree~$T$ and a map $\varphi:T\longrightarrow T$, if we denote by $\Fc$ the set of all functions $f:T\longrightarrow \CC$, we have that $C_{\varphi}:\Fc\longrightarrow\Fc$ defined as $C_{\varphi}(f):=f \circ \varphi$ for each $f \in \Fc$ is a linear map. This map~$C_{\varphi}$ is called the {\em composition operator} with symbol~$\varphi$. By the Closed Graph Theorem $C_{\varphi}$ is bounded in~$\Lc(T)$ if and only if $C_{\varphi}(\Lc(T)) \subset \Lc(T)$. By \cite[Theorem~3.2]{ACoEas2014} this holds when~$\varphi$ is {\em Lipschitz}, i.e.\ when
\[
L_{\varphi} := \sup_{v\neq w \in T} \frac{\dist(\varphi(v),\varphi(w))}{\dist(v,w)} < \infty,
\]
where $L_{\varphi}$ is said to be the {\em Lipschitz constant} of $\varphi$. The Closed Graph Theorem shows again that $C_{\varphi}$ is bounded in $\Lc_0(T)$ if and only if $C_{\varphi}(\Lc_0(T)) \subset \Lc_0(\Lc_0(T))$. However, the fact that $\varphi$ is Lipschitz is not sufficient to ensure this last inclusion (though the Lipschitz condition is necessary) and a complete characterization of when $C_{\varphi}$ is bounded in $\Lc_0(T)$ was established in \cite[Section~3]{ColonnaMar2025_MJM_composition}. When $C_{\varphi}$ is a bounded operator, the following observation is a direct consequence of Remark~\ref{Rem:1-isom}:

\begin{remark}[\textbf{A useful conjugacy}]\label{Rem:2-conjugacy}
	Given a rooted tree $T$ and a map $\varphi:T\longrightarrow T$ such that~$C_{\varphi}$ is a bounded operator in $\Lc(T)$, then one can find a natural $\ell^{\infty}(T)$-operator that is {\em topologically conjugate} to $C_{\varphi}$ (see \cite[Section~1.1]{GrPe2011_book} for the precise definition of ``conjugacy''). In fact, letting
	\[
	\hat{C_{\varphi}}(x) := (D\circ C_{\varphi}\circ D^{-1})(x) \quad \text{ for each } x \in \ell^{\infty}(T),
	\]
	where the map $D$ is the isomorphism established in Remark~\ref{Rem:1-isom}, we have that the diagram
	\begin{equation*}
		\begin{CD}
			\Lc(T) @>{C_{\varphi}}>> \Lc(T) \\
			@V{D}VV  @VV{D}V \\
			\ell^{\infty}(T) @>{\hat{C_{\varphi}}}>> \ell^{\infty}(T)
		\end{CD}
	\end{equation*}
	is commutative, so that $C_{\varphi}$ and $\hat{C_{\varphi}}$ are conjugate via the isometric isomorphism $D$. Note that the same still holds if we replace $\Lc(T)$ and $\ell^{\infty}(T)$ with the respective spaces $\Lc_0(T)$ and $c_0(T)$, when $C_{\varphi}$ is bounded in $\Lc_0(T)$. Thus, these operators share all the dynamical properties that are preserved under conjugacy such as hypercyclicity, density of periodic points and Devaney chaos among others. 
\end{remark}

Using Remarks~\ref{Rem:1-isom} and \ref{Rem:2-conjugacy} we can improve \cite[Theorem~3.7]{ColonnaMar2025_MJM_composition}:

\begin{corollary}\label{Cor:norm}
	Let $\varphi$ be a self-map of a rooted tree $T$. If the composition operator $C_{\varphi}$ is bounded in the space $\Lc(T)$, then its operator norm fulfills the equality $\|C_{\varphi}\|_{\Lc} = \max\{ 1+|\varphi(\zero)| , L_{\varphi} \}$, where $L_{\varphi}$ is the Lipschitz constant of $\varphi$. Moreover, if $C_{\varphi}$ is bounded in the space $\Lc_0(T)$, then the same equality holds for its respective norm as an operator on $\Lc_0(T)$.
\end{corollary}
\begin{proof}
	By \cite[Theorem~3.7]{ColonnaMar2025_MJM_composition} we already have the following inequalities
	\[
	1+|\varphi(\zero)| \leq \|C_{\varphi}\|_{\Lc} \leq \max\{ 1+|\varphi(\zero)| , L_{\varphi} \},
	\]
	so it is enough to check that $\|C_{\varphi}\|_{\Lc}\geq L_{\varphi}$ when $L_{\varphi}\geq 1$. We will prove that $\|\hat{C_{\varphi}}\|_{\infty}\geq L_{\varphi}$, where~$\hat{C_{\varphi}}$ is the respective operator obtained from Remark~\ref{Rem:2-conjugacy}. This would complete the proof because both operators are topologically conjugated via the isometric isomorphism $D$ from Remark~\ref{Rem:1-isom}.
	
	By \cite[Proposition~3.1]{ACoEas2014} we can find a vertex $u \in T$ fulfilling that $\dist(\varphi(u),\varphi(u^-))=L_{\varphi}\geq 1$. Thus, let $[\zero,u_1,u_2,...,\varphi(u)]$ and $[\zero,v_1,v_2,...,\varphi(u^-)]$ be the unique paths in $T$ between the root $\zero$ and each of the vertices $\varphi(u)$ and $\varphi(u^-)$, and consider the subsets $U,V \subset T$ with $U:=\{ \zero, u_1, u_2, ..., \varphi(u) \}$ and $V:=\{ \zero, v_1, v_2, ..., \varphi(u^-) \}$. Consider $x=(x_v)_{v\in T} \in \CC^{T}$ defined as
	\[
	x_v :=
	\begin{cases}
		\phantom{-}1, & \text{ if } v \in U\setminus V, \\
		-1, & \text{ if } v \in V\setminus U, \\
		\phantom{-}0, & \text{ otherwise}.
	\end{cases}
	\]
	It is clear that $x \in c_{00}(T) \subset c_0(T) \subset \ell^{\infty}(T)$, and that $\|x\|_{\infty}=1$ because both $U\setminus V$ and $V\setminus U$ can not be empty at the same time (recall that $\varphi(u)\neq\varphi(u^-)$). We claim that $\|\hat{C_{\varphi}}(x)\|_{\infty}\geq L_{\varphi}$, which would complete the proof. In fact, from Remarks~\ref{Rem:1-isom}~and~\ref{Rem:2-conjugacy} we have that
	\[
	[(C_{\varphi} \circ D^{-1})(x)](u) = \sum_{v \in [\zero,u_1,u_2,...,\varphi(u)]} x_v \quad \text{ and } \quad [(C_{\varphi} \circ D^{-1})(x)](u^-) = \sum_{v \in [\zero,v_1,v_2,...,\varphi(u^-)]} x_v,
	\]
	and, since $\hat{C_{\varphi}}(x) := (D \circ C_{\varphi} \circ D^{-1})(x)$, we have that
	\[
	[\hat{C_{\varphi}}(x)]_u = \left( \sum_{v \in [\zero,u_1,u_2,...,\varphi(u)]} x_v \right) - \left( \sum_{v \in [\zero,v_1,v_2,...,\varphi(u^-)]} x_v \right) = \left( \sum_{v \in U\setminus V} 1 \right) - \left( \sum_{v \in V\setminus U} -1 \right).
	\]
	Since the unique path between $\varphi(u)$ and $\varphi(u^-)$ is formed by the vertices in $(U\setminus V)\cup (V\setminus U)$ and a common ancestor of $\varphi(u)$ and $\varphi(u^-)$, it is not hard to check the cardinal equality
	\[
	\text{card}(U\setminus V)+\text{card}(V\setminus U) = \dist(\varphi(u),\varphi(u^-)).
	\]
	Hence, $[\hat{C_{\varphi}}(x)]_u = \text{card}(U\setminus V)+\text{card}(V\setminus U) = \dist(\varphi(u),\varphi(u^-)) = L_{\varphi}$, and the result follows.
\end{proof}

The next example shows an operator from the class of operators considered in our main result:

\begin{example}\label{Exa:NN_0}
	Let $T$ be indexed by $\NN_0$, with $0 \in \NN_0$ as its root and with $k,l \in \NN_0$ being adjacent if and only if $|k-l|=1$. Consider $\varphi:\NN_0\longrightarrow\NN_0$ defined as $\varphi(j):=j+1$ for every $j \in \NN_0$ and note that the respective composition operator $C_{\varphi}$ is bounded in $\Lc_0(\NN_0)$ by \cite[Theorem~3.1]{ColonnaMar2025_MJM_composition}. Given any $x=(x_j)_{j\in\NN_0} \in c_0(\NN_0)$, using the map $D$ from Remark~\ref{Rem:1-isom} for this particular case we have that
	\[
	[D^{-1}(x)](j) = \sum_{k=0}^{j} x_k \quad \text{ and hence } \quad [(C_{\varphi} \circ D^{-1})(x)](j) = \sum_{k=0}^{j+1} x_k \quad \text{ for each } j \in \NN_0.
	\]
	Finally, from the conjugacy established in Remark~\ref{Rem:2-conjugacy} but for this particular case, we have that
	\[
	\hat{C_{\varphi}}(x) := (D \circ C_{\varphi} \circ D^{-1})(x) = (x_0+x_1,x_2,x_3,x_4,...).
	\]
	Hence $\hat{C_{\varphi}}=P_0+B$, where $P_0$ is the projection on the $0$-coordinate and where $B$ is the classical backward shift in $c_0(\NN_0)$ as denoted at the Introduction for the Rolewicz operators. This operator admits a representation as an infinite-matrix $A_{\varphi} = (a_{i,j})_{i,j\in\NN_0}$ when acting on $c_0(\NN_0)$ as
	\[
	A_{\varphi} =
	\begin{matrix}
		\text{\tiny $i=0$} \\ \text{\tiny $i=1$} \\ \text{\tiny $i=2$} \\ \vdots \\ \text{\tiny $i$} \\ \vdots
	\end{matrix}
	\begin{pmatrix}
		\overset{\text{\tiny $j=0$}}{1} & \overset{\text{\tiny $j=1$}}{1} & \overset{\text{\tiny $j=2$}}{0} & \overset{\text{\tiny $j=3$}}{0} & 0 & \overset{\text{\tiny $j=i+1$}}{0} & 0 & \cdots \\
		0 & 0 & 1 & 0 & 0 & 0 & 0 & \cdots \\
		0 & 0 & 0 & 1 & 0 & 0 & 0 & \cdots \\
		\vdots & \vdots & \vdots & \vdots & \ddots & \vdots & \vdots & \vdots \\
		0 & 0 & 0 & 0 & 0 & 1 & 0 & \cdots \\
		\vdots & \vdots & \vdots & \vdots & \vdots & \vdots & \ddots & \vdots
	\end{pmatrix},
	\]
	with two consecutive $1$'s in the $0$-row, and then a $1$ in the column $j=i+1$ of each $i$-th row for $i \in \NN$.
\end{example}

\subsection{Multiplication and the backward shift operators}\label{SubSec_2.3:mult_shift}

The main objective of this paper is to exhibit simple examples of frequently hypercyclic composition operators on the little Lipschitz space of a rooted tree, and the reader can directly go to Section~\ref{Sec_3:existence} to find these operators. However, we have decided to include this subsection, where we use the arguments from Remark~\ref{Rem:2-conjugacy} for the multiplication and backward shift operators on $\Lc(T)$ and $\Lc_0(T)$ as defined in \cite{ACoEas2011,ACoEas2012,ACoEas2013,CoEas2010,CoEas2012,MarRi2020}. Following the spirit of Corollary~\ref{Cor:norm}, with our ``isomorphic approach'' we can compute the precise norm of the mentioned operators. In particular, we improve \cite[Theorem~4.1]{CoEas2010} via Corollary~\ref{Cor:norm_mult}, and then we solve all the questions left in \cite[Section~4]{MarRi2020} via Corollary~\ref{Cor:norm_shift} below.

\begin{remark}[\textbf{Multiplication operators}]\label{Rem:mult}
	As in \cite{CoEas2010}, given a rooted tree~$T$ and a complex-valued map $\psi:T\longrightarrow \CC$, if we again denote by $\Fc$ the set of all complex-valued functions $f:T\longrightarrow \CC$ then we can consider the linear map $M_{\psi}:\Fc\longrightarrow\Fc$ defined as $M_{\psi}(f):=\psi \cdot f$ for each $f \in \Fc$. Such a map~$M_{\psi}$ will be called the {\em multiplication operator} with symbol~$\psi$.
	
	By \cite[Theorem~3.6]{CoEas2010} we know that $M_{\psi}$ is well-defined and a bounded operator in $\Lc(T)$, if and only if $M_{\psi}$ is well-defined and bounded operator in $\Lc_0(T)$, if and only if
	\[
	\sup_{v \in T} |\psi(v)| < \infty \quad \text{ and } \quad \sup_{v \in T^*} \left( |v| \cdot |\psi(v)-\psi(v^-)| \right) < \infty.
	\]
	As in Remark~\ref{Rem:2-conjugacy}, if $M_{\psi}$ is a bounded operator in $\Lc(T)$, then one can find a natural $\ell^{\infty}(T)$-operator that is {\em topologically conjugate} to $M_{\psi}$ by letting
	\[
	\hat{M_{\psi}}(x) := (D\circ M_{\psi}\circ D^{-1})(x) \quad \text{ for each } x \in \ell^{\infty}(T),
	\]
	where the map $D$ is again the isomorphism established in Remark~\ref{Rem:1-isom}. We thus have that the diagram
	\begin{equation*}
		\begin{CD}
			\Lc(T) @>{M_{\psi}}>> \Lc(T) \\
			@V{D}VV  @VV{D}V \\
			\ell^{\infty}(T) @>{\hat{M_{\psi}}}>> \ell^{\infty}(T)
		\end{CD}
	\end{equation*}
	is commutative, so that $M_{\psi}$ and $\hat{M_{\psi}}$ are conjugate via the isometric isomorphism $D$. The same still holds if we replace $\Lc(T)$ and $\ell^{\infty}(T)$ with the respective spaces $\Lc_0(T)$ and $c_0(T)$.
\end{remark}

\begin{corollary}\label{Cor:norm_mult}
	Let $\psi:T\longrightarrow\CC$ be a complex-valued map on a rooted tree $T$. If the multiplication operator $M_{\psi}$ is bounded in the space $\Lc(T)$, or equivalently in $\Lc_0(T)$, then its operator norm is
	\[
	\|M_{\psi}\|_{\Lc} = \max\left\{ |\psi(\zero)| , \sup_{v \in T^*} \left( |\psi(v)| + |v| \cdot |\psi(v)-\psi(v^-)| \right) \right\}.
	\]
\end{corollary}
\begin{proof}
	As in Corollary~\ref{Cor:norm} we will directly compute the norm of $\|\hat{M_{\psi}}\|_{\infty}$, which coincides with $\|M_{\psi}\|_{\Lc}$ by the isometry condition of $D$. Firstly, note that given any $x=(x_v)_{v\in T} \in \ell^{\infty}(T)$ we have
	\[
	[(M_{\psi} \circ D^{-1})(x)](v) = \psi(v) \cdot \left( \sum_{w \in [\zero,v_1,v_2,...,v]} x_w \right) \quad \text{ for every } v \in T.
	\]
	Thus, since $\hat{M_{\psi}}(x) := (D \circ M_{\psi} \circ D^{-1})(x)$, we have that $[\hat{M_{\psi}}(x)]_{\zero} = \psi(\zero) \cdot x_{\zero}$, and that
	\begin{align}\label{eq:M}
		[\hat{M_{\psi}}(x)]_v &= \psi(v) \cdot \left( \sum_{w \in [\zero,v_1,v_2,...,v]} x_w \right) - \psi(v^-) \cdot \left( \sum_{w \in [\zero,v_1,v_2,...,v^-]} x_w \right) \nonumber \\[7.5pt]
		&= \psi(v) \cdot x_v + \left( \psi(v)-\psi(v^-) \right) \cdot \left( \sum_{w \in [\zero,v_1,v_2,...,v^-]} x_w \right),
	\end{align}
	for every $v \in T^*$. Taking modulus on the previous expressions and using the triangle inequality in a standard way one can easily show, when $\|x\|_{\infty} \leq 1$, that
	\[
	\left| [\hat{M_{\psi}}(x)]_v \right| \leq \max\left\{ |\psi(\zero)| , \sup_{v \in T^*} \left( |\psi(v)| + |v| \cdot |\psi(v)-\psi(v^-)| \right) \right\} \quad \text{ for all } v \in T.
	\]
	Thus,
	\[
	\|\hat{M_{\psi}}\|_{\infty} = \sup_{\|x\|_{\infty} \leq 1} \| \hat{M_{\psi}}(x) \|_{\infty} \leq \max\left\{ |\psi(\zero)| , \sup_{v \in T^*} \left( |\psi(v)| + |v| \cdot |\psi(v)-\psi(v^-)| \right) \right\}.
	\]
	Conversely, if for each $u \in T^*$ we have that $[\zero,u_1,u_2,...,u]$ is the unique path from $\zero$ to $u$, then we can construct a vector $x^{(u)} \in c_0(T) \subset \ell^{\infty}(T)$ as follows:
	\begin{enumerate}[--]
		\item let $x^{(u)}_{\zero}=x^{(u)}_{u_1}=x^{(u)}_{u_2}=...=x^{(u)}_{u^-}=\lambda \in \CC$ with $|\lambda|=1$ and $\lambda \cdot \left( \psi(u)-\psi(u^-) \right) = \left| \psi(u)-\psi(u^-) \right|$;
		
		\item pick $x^{(u)}_u \in \CC$ with $|x^{(u)}_u|=1$ and such that $\psi(u) \cdot x^{(u)}_u = |\psi(u)|$;
		
		\item and let $x^{(u)}_w = 0$ for all $w \in T \setminus \{\zero,u_1,u_2,...,u\}$.
	\end{enumerate}
	Then we have that $\|x^{(u)}\|_{\infty}=1$, but also that $\left| [\hat{M_{\psi}}(x^{(u)})]_{\zero} \right| = |\psi(\zero)|$, and by \eqref{eq:M} that
	\[
	[\hat{M_{\psi}}(x^{(u)})]_u = \psi(u) \cdot x^{(u)}_u + \left( \psi(u)-\psi(u^-) \right) \cdot |u| \cdot \lambda = |\psi(u)| + |u| \cdot |\psi(u)-\psi(u^-)|,
	\] 
	for each $u \in T^*$. Thus,
	\[
	\|\hat{M_{\psi}}\|_{\infty} \geq \sup_{u\in T^*} \|\hat{M_{\psi}}(x^{(u)})\|_{\infty} \geq \max\left\{ |\psi(\zero)| , \sup_{u \in T^*} \left( |\psi(u)| + |u| \cdot |\psi(u)-\psi(u^-)| \right) \right\},
	\]
	finishing the proof.
\end{proof}

In \cite{CoEas2010}, the equivalent norm $\|\cdot\|$ defined in \eqref{eq:norm_+} was used instead than $\|\cdot\|_{\Lc}$ defined in \eqref{eq:norm_L}. The reader should be careful since in \cite{CoEas2010} the symbol ``$\|\cdot\|_{\Lc}$'' is used to denote the norm $\|\cdot\|$. However, also when one uses $\|\cdot\|$, similar computations to those employed in Corollary~\ref{Cor:norm_mult} lead to
\[
\|M_{\psi}\| = |\psi(0)| + \sup_{v \in T^*} \left( |\psi(v)| + |v| \cdot |\psi(v)-\psi(v^-)| \right)
\]
for the norm of $M_{\psi}$ as an operator from $(\Lc(T),\|\cdot\|)$ to itself, or from $(\Lc_0(T),\|\cdot\|)$ to itself. This last equality is our real improvement of \cite[Theorem~4.1]{CoEas2010}. The ``isomorphic approach'' developed here may be used to compute the exact norm of the multiplication operators considered in \cite{ACoEas2011,ACoEas2012,ACoEas2013,CoEas2012}.

\begin{remark}[\textbf{The backward shift}]\label{Rem:shift}
	As in \cite{MarRi2020}, given a rooted tree~$T$, if we again denote by $\Fc$ the set of all functions $f:T\longrightarrow \CC$, we can consider the linear map $B:\Fc\longrightarrow\Fc$ defined for $f \in \Fc$ as
	\[
	[B(f)](v) := \sum_{w \in \Chi(v)} f(w) \quad \text{ for each } v \in T.
	\]
	This map is called the {\em backward shift} on $\Fc$. The reader must distinguish this map from the original backward shift mentioned at the Introduction of this paper (these operators act in different spaces).
	
	Following \cite{MarRi2020}, given a vertex $v \in T$ we will denote by $\gamma(v)$ the cardinality of the set $\Chi(v)$ of children of $v$. Moreover, a rooted tree~$T$ is said to be {\em homogeneous by sectors} ({\em at level $N$}) if there exists $N \in \NN_0$ such that for all $v \in T$ with $|v|=N$ we have that $\gamma(v)=\gamma(w)$ for every $w \in S_{v}$. In this case, by \cite[Theorem~4.10]{MarRi2020}, we know that {\em the following statements are equivalent:
	\begin{enumerate}[{\em(i)}]
		\item $B$ is well-defined and a bounded operator from $\Lc(T)$ to itself;
		
		\item $B$ is well-defined and a bounded operator from $\Lc_0(T)$ to itself;
		
		\item the tree $T$ is homogeneous by sectors;
		
		\item we have that $\sup\left\{ |v| \cdot |\gamma(v)-\gamma(v^-)| \ ; \ v\in T^* \right\} < \infty$.
	\end{enumerate}
	In addition, when these conditions hold one has that $\|B\|_{\Lc} \leq \max\left\{ 2\gamma(\zero) , \Lambda(T) \right\}$, where
	\[
	\Lambda(T) := \sup_{v \in T^*} \left\{ \gamma(v) + \gamma(v^-) - 1 + \left|\gamma(v)-1\right| + |v| \cdot |\gamma(v)-\gamma(v^-)| \right\}.
	\]}
	
	As in Remarks~\ref{Rem:2-conjugacy}~and~\ref{Rem:mult}, when $B$ is bounded in $\Lc(T)$, then one can find a natural $\ell^{\infty}(T)$-operator that is {\em topologically conjugate} to $B$ by letting
	\[
	\hat{B}(x) := (D\circ B\circ D^{-1})(x) \quad \text{ for each } x \in \ell^{\infty}(T).
	\]
	Here, the map $D$ is again the isomorphism established in Remark~\ref{Rem:1-isom}. We thus have, one more time, that the diagram
	\begin{equation*}
		\begin{CD}
			\Lc(T) @>{B}>> \Lc(T) \\
			@V{D}VV  @VV{D}V \\
			\ell^{\infty}(T) @>{\hat{B}}>> \ell^{\infty}(T)
		\end{CD}
	\end{equation*}
	is commutative, so that $B$ and $\hat{B}$ are conjugate via the isometric isomorphism $D$. The same still holds if we replace $\Lc(T)$ and $\ell^{\infty}(T)$ with the respective spaces $\Lc_0(T)$ and $c_0(T)$. In particular, given any vector $x=(x_v)_{v\in T} \in \ell^{\infty}(T)$ we have for each $v \in T$ that
	\[
	[(B \circ D^{-1})(x)](v) = \sum_{u \in \Chi(v)} \left( \sum_{w \in [\zero,v_1,v_2,...,v,u]} x_w \right) = \gamma(v) \cdot \left( \sum_{w \in [\zero,v_1,v_2,...,v]} x_w \right) + \sum_{u \in \Chi(v)} x_u.
	\]
	Thus, since $\hat{B}(x) := (D \circ B \circ D^{-1})(x)$, we have that
	\begin{equation}\label{eq:bs_zero}
		[\hat{B}(x)]_{\zero} = \gamma(\zero) \cdot x_{\zero} + \sum_{w \in \Chi(\zero)} x_w,
	\end{equation}
	and that
	\begin{align}\label{eq:bs_v}
		[\hat{B}(x)]_v &= \left( \gamma(v) \cdot \left( \sum_{w \in [\zero,v_1,...,v]} x_w \right) + \sum_{w \in \Chi(v)} x_w \right) - \left( \gamma(v^-) \cdot \left( \sum_{w \in [\zero,v_1,...,v^-]} x_w \right) + \sum_{w \in \Chi(v^-)} x_w \right) \nonumber \\[7.5pt]
		&=  \sum_{w \in \Chi(v)} x_w \ \ - \sum_{w \in \Chi(v^-)\setminus\{v\}} x_w + (\gamma(v)-1) x_v + \left( \gamma(v)-\gamma(v^-) \right) \cdot \left( \sum_{w \in [\zero,v_1,...,v^-]} x_w \right),
	\end{align}
	for each $v \in T^*$.
\end{remark}

\begin{corollary}\label{Cor:norm_shift}
	Let $T$ be a rooted tree. If the backward shift operator $B$ is bounded in the space $\Lc(T)$, or equivalently in $\Lc_0(T)$, then its operator norm in these spaces is $\|B\|_{\Lc} = \max\left\{ 2\gamma(\zero) , \Lambda(T) \right\}$.
\end{corollary}
\begin{proof}
	As in Corollaries~\ref{Cor:norm}~and~\ref{Cor:norm_mult} we will directly compute the norm of $\|\hat{B}\|_{\infty}$, which coincides with that of $\|B\|_{\Lc}$ by the isometry condition of $D$. By \cite[Theorem~4.10]{MarRi2020}, we already know that
	\[
	\|\hat{B}\|_{\infty} \leq \max\left\{ 2\gamma(\zero) , \Lambda(T) \right\}.
	\]
	For the converse we start considering the vector $x^{(\zero)} \in c_0(T) \subset \ell^{\infty}(T)$ defined as $x^{(\zero)}_w = 1$ when $w \in \{\zero\} \cup \Chi(\zero)$ and $x^{(\zero)}_w=0$ otherwise. It follows that $\|x^{(\zero)}\|_{\infty}=1$, and from \eqref{eq:bs_zero} we have that
	\[
	[\hat{B}(x^{(\zero)})]_{\zero} = \gamma(\zero) + \sum_{w \in \Chi(\zero)} 1 = 2\gamma(\zero).
	\]
	Moreover, if for each $u \in T^*$ we have that $[\zero,u_1,u_2,...,u]$ is the unique path from $\zero$ to $u$ then we can construct a vector $x^{(u)} \in c_0(T) \subset \ell^{\infty}(T)$ as follows:
	\begin{enumerate}[--]
		\item let $x^{(u)}_{\zero}=x^{(u)}_{u_1}=x^{(u)}_{u_2}=...=x^{(u)}_{u^-}=
		\begin{cases}
			\phantom{-}1 & \text{ if } \gamma(u)-\gamma(u^-)\geq 0, \\[1pt]
			-1 & \text{ if } \gamma(u)-\gamma(u^-)<0;
		\end{cases}$\\[-7.5pt]
		
		\item let $x^{(u)}_{u}=
		\begin{cases}
			\phantom{-}1 & \text{ if } \gamma(u)-1\geq 0, \\[1pt]
			-1 & \text{ if } \gamma(u)-1<0;
		\end{cases}$\\[-5pt]
		
		\item let $x^{(u)}_w=-1$ for all $w \in \Chi(u^-)\setminus\{u\}$;
		
		\item let $x^{(u)}_w=1$ for all $w \in \Chi(u)$;
		
		\item and let $x^{(u)}_w = 0$ for all $w \in T \setminus \left( \{\zero,u_1,u_2,...,u\} \cup \Chi(u^-) \cup \Chi(u) \right)$.
	\end{enumerate}
	It follows that $\|x^{(u)}\|_{\infty}=1$ and, by \eqref{eq:bs_v}, that
	\begin{align*}
		[\hat{B}(x^{(u)})]_u &= \left( \sum_{w \in \Chi(u)} 1 \right) - \left( \sum_{w \in \Chi(u^-)\setminus\{u\}} -1 \right) + |\gamma(u)-1| +  |\gamma(u)-\gamma(u^-)| \cdot |u| \\[10pt]
		&= \gamma(u) + \gamma(u^-) - 1 + \left|\gamma(u)-1\right| +  |u| \cdot | \gamma(u)-\gamma(u^-) |,
	\end{align*}
	for each $u \in T^*$. Thus,
	\begin{align*}
		\|\hat{B}\|_{\infty} &\geq \sup_{u\in T} \|\hat{B}(x^{(u)})\|_{\infty} = \max\left\{ \|\hat{B}(x^{(\zero)})\|_{\infty} , \sup_{u\in T^*} \|\hat{B}(x^{(u)})\|_{\infty} \right\} \\[5pt]
		&\geq \max\left\{ 2\gamma(\zero) , \sup_{u \in T^*} \left\{ \gamma(u) + \gamma(u^-) - 1 + \left|\gamma(u)-1\right| + |u| \cdot |\gamma(u)-\gamma(u^-)| \right\} \right\} \\[7.5pt]
		&= \max\left\{ 2\gamma(\zero) , \Lambda(T) \right\}.\qedhere
	\end{align*}
\end{proof}

In \cite{MarRi2020} the underlying tree $T$ was allowed to have \textit{leaves} (i.e.\ vertices with no children). However, the proof of Corollary~\ref{Cor:norm_shift} remains valid for trees with leaves since the inequality ``$\gamma(u)-1<0$'' has been considered. Moreover, Corollary~\ref{Cor:norm_shift} resolves all the questions raised in \cite[Section~4]{MarRi2020} concerning the sharpness of the estimates for the norm $\|B\|_{\Lc}$ obtained there. For comparison, see \cite{ZhangZhao2022_MJM_norm}, where this result is proved by a considerably more involved argument than the one presented here.

\section[Frequently hypercyclic composition operators]{Frequently hypercyclic composition operators on $\Lc_0(T)$}\label{Sec_3:existence}

In this section we address Question~\ref{Q:first} for the particular case of composition operators whose symbols are strictly increasing self-maps of the set of non-negative integers (see Theorem~\ref{The:NN_0} below). Recall first that given a bounded linear operator $A:X\longrightarrow X$ on a separable Banach space $X$ we say that:
\begin{enumerate}[--]
	\item the operator $A$ is {\em hypercyclic} if there exists $x \in X$ for which the set $\{ A^n(x) \ ; \ n\in\NN \}$ is dense in $X$;
	
	\item the operator $A$ {\em satisfies the Frequent Hypercyclicity Criterion} if there exist a dense subset $X_0 \subset X$ and a (not necessarily linear nor continuous) map $B:X_0\longrightarrow X_0$ such that, for any $x \in X_0$,
	\begin{enumerate}[(I)]
		\item $\sum_{n=0}^{\infty} A^n(x)$ converges unconditionally;
		
		\item $\sum_{n=0}^{\infty} B^n(x)$ converges unconditionally;
		
		\item $AB(x)=x$, i.e.\ $B$ is a right-inverse for $A$.
	\end{enumerate}
\end{enumerate}

All the hypercyclic operators that we are going to consider in this paper do satisfy this version of the Frequent Hypercyclicity Criterion, which is one of the strongest properties (dynamically speaking) that a linear operator may present. In fact, let us recall that satisfying the Frequent Hypercyclicity Criterion implies being {\em frequently hypercyclic} (and hence hypercyclic), and that with the form presented here it also implies being {\em Devaney chaotic} and {\em topologically mixing} (see \cite[Proposition~9.11]{GrPe2011_book}), having the so-called {\em operator specification property} (see \cite{BarMarPe2016}), but also admitting {\em ergodic and invariant measures with full support} (see \cite{Agneessens2023} but also \cite{GriLo2023_JMPA,Lopez2024_RinM}), among many other strong dynamical properties.

We are not going to give the precise definitions of the dynamical notions mentioned above, and we refer the reader to the textbooks \cite{BaMa2009_book,GrPe2011_book} for more on Linear Dynamics. Actually, with the previous paragraph we only wanted to remark that the operators we are about to consider present a much more interesting dynamical behaviour than just hypercyclicity. Recall, though, that the main result of this section came originally motivated by Question~\ref{Q:first}, where only hypercyclicity was considered:

\begin{theorem}\label{The:NN_0}
	Let $\varphi:\NN_0\longrightarrow\NN_0$ be a strictly increasing map with $\varphi(0)>0$ and assume that the composition operator $C_{\varphi}:\Lc_0(T)\longrightarrow\Lc_0(T)$ is bounded, where $T$ is the rooted tree indexed by $\NN_0$, with $0 \in \NN_0$ as its root and with two vertex $k,l \in \NN_0$ being adjacent whenever $|k-l|=1$. Hence, the following statements are equivalent:
	\begin{enumerate}[{\em(i)}]
		\item $C_{\varphi}$ satisfies the Frequent Hypercyclicity Criterion;
		
		\item $C_{\varphi}$ is hypercyclic;
		
		\item $\displaystyle \limsup_{n\to\infty} \ \dist(\varphi^n(j),\varphi^n(j-1)) = \infty$ for every $j \in \NN$;
		
		\item $\displaystyle \lim_{n\to\infty} \ \dist(\varphi^n(j),\varphi^n(j-1)) = \infty$ for every $j \in \NN$;
		
		\item for each $j \in \NN$ there exists some $n_j \in \NN$ such that $\dist(\varphi^{n_j}(j),\varphi^{n_j}(j-1))\geq 2$.
	\end{enumerate}
\end{theorem}
Before starting the proof recall that with the notation from Section~\ref{Sec_2:preliminar}, but for the tree considered in the statement of Theorem~\ref{The:NN_0}, we have the equality $j^-=j-1$ for each $j \in \NN$.
\begin{proof}[Proof of Theorem~\ref{The:NN_0}]
	The implication (i) $\Rightarrow$ (ii) is well-known (see \cite[Chapter~9]{GrPe2011_book}). For (ii) $\Rightarrow$ (iii) note that, if $C_{\varphi}$ is hypercyclic, then it follows from \cite[Proposition~5.4]{ColonnaMar2025_MJM_composition} that for each $j \in \NN$ the set of distances $\{ \dist(\varphi^n(j),\varphi^n(j^-)) \ ; \ n \in \NN_0 \}$ is unbounded. This fact trivially implies
	\[
	\limsup_{n\to\infty} \ \dist(\varphi^n(j),\varphi^n(j^-)) = \infty \quad \text{ for every } j \in \NN. 
	\]
	Let us prove (iii) $\Leftrightarrow$ (iv) $\Leftrightarrow$ (v) using that $\varphi$ is strictly increasing. Actually, the increasing condition implies that $\dist(\varphi^n(j),\varphi^n(j^-)) \leq \dist(\varphi^{n+1}(j),\varphi^{n+1}(j^-))$ for every pair $n,j \in \NN$, so that
	\[
	\limsup_{n\to\infty} \ \dist(\varphi^n(j),\varphi^n(j^-)) = \lim_{n\to\infty} \ \dist(\varphi^n(j),\varphi^n(j^-)) \quad \text{ for every } j \in \NN,
	\]
	and (iv) holds as soon as (iii) does. The implication (iv) $\Rightarrow$ (v) is trivial. Finally, if (v) holds we can briefly justify that (iii) also holds as follows: fix any arbitrary index $j \in \NN$, name it as $j_0:=j$ and note that by (v) there is $n_{j_0} \in \NN$ such that $\dist(\varphi^{n_{j_0}}(j_0),\varphi^{n_{j_0}}(j_0^-)) \geq 2$; letting $j_1:=\varphi^{n_{j_0}}(j_0)$, noticing that $\varphi^{n_{j_0}}(j_0-1) < j_1^-$, and since $\dist(\varphi^{n_{j_1}}(j_1),\varphi^{n_{j_1}}(j_1^-)) \geq 2$ for some $n_{j_1} \in \NN$ by (v), then
	\begin{align*}
		\dist\left(\varphi^{n_{j_1}+n_{j_0}}(j_0),\varphi^{n_{j_1}+n_{j_0}}(j_0^-)\right) &= \dist\left(\varphi^{n_{j_1}}\left(\varphi^{n_{j_0}}(j_0)\right),\varphi^{n_{j_1}}\left(\varphi^{n_{j_0}}(j_0^-)\right)\right) \\[5pt]
		&= \dist\left(\varphi^{n_{j_1}}(j_1),\varphi^{n_{j_1}}(j_1^-)\right) + \dist\left(\varphi^{n_{j_1}}(j_1^-),\varphi^{n_{j_1}}\left(\varphi^{n_{j_0}}(j_0^-)\right)\right) \\[5pt]
		&\geq 2 + 1 = 3.
	\end{align*}
	A trivial recursive argument, in which one lets $j_{k+1}:=\varphi^{\sum_{l=0}^{k} n_{j_l}}(j_0)$ for each $k \in \NN$, shows the existence of a sequence $(j_k)_{k\in\NN_0} \in \NN^{\NN_0}$ and by (v) a corresponding sequence $(n_{j_k})_{k\in\NN_0} \in \NN^{\NN_0}$ fulfilling that
	\begin{align*}
		\dist\left(\varphi^{\sum_{l=0}^k n_{j_l}}(j_0),\varphi^{\sum_{l=0}^k n_{j_l}}(j_0^-)\right) &= \dist\left(\varphi^{n_{j_k}}\left(\varphi^{\sum_{l=0}^{k-1} n_{j_l}}(j_0)\right),\varphi^{n_{j_k}}\left(\varphi^{\sum_{l=0}^{k-1} n_{j_l}}(j_0^-)\right)\right) \\[5pt]
		&= \dist\left(\varphi^{n_{j_k}}(j_k),\varphi^{n_{j_k}}(j_k^-)\right) + \dist\left(\varphi^{n_{j_k}}(j_k^-),\varphi^{n_{j_k}}\left(\varphi^{\sum_{l=0}^{k-1} n_{j_l}}(j_0^-)\right)\right) \\[5pt]
		&\geq 2 + k  \quad \text{ for each } k \in \NN_0.
	\end{align*}
	Thus, $\limsup_{n\to\infty} \ \dist(\varphi^n(j_0),\varphi^n(j_0^-)) = \infty$. The arbitrariness of $j_0=j \in \NN$ implies (iii).
	
	To finish the proof of Theorem~\ref{The:NN_0} it is then enough to show (iv) $\Rightarrow$ (i). Assume~(iv) and let us show that then the respective operator $\hat{C_{\varphi}}:c_0(\NN_0)\longrightarrow c_0(\NN_0)$, obtained from Remarks~\ref{Rem:1-isom} and \ref{Rem:2-conjugacy} applied to the tree $T=\NN_0$ and the map $\varphi$ described in the statement of Theorem~\ref{The:NN_0}, satisfies the previously introduced Frequent Hypercyclicity Criterion. We start noticing that given any $x=(x_j)_{j\in\NN_0} \in c_0(\NN_0)$, using the map $D$ from Remark~\ref{Rem:1-isom} as in Example~\ref{Exa:NN_0} but for this general case, we have that
	\[
	[D^{-1}(x)](j) = \sum_{k=0}^{j} x_k \quad \text{ and then } \quad [(C_{\varphi} \circ D^{-1})(x)](j) = \sum_{k=0}^{\varphi(j)} x_k \quad \text{ for each } j \in \NN_0.
	\]
	Hence, from the conjugacy established in Remark~\ref{Rem:2-conjugacy} we have that
	\begin{equation}\label{eq:C}
		\hat{C_{\varphi}}(x) := (D \circ C_{\varphi} \circ D^{-1})(x) = \left( \sum_{k=\varphi(j^-)+1}^{\varphi(j)} x_k \right)_{j\in\NN_0},
	\end{equation}
	where by abuse of notation we are setting $\varphi(0^-):=-1$. Moreover, a similar calculation also shows that for any $x=(x_j)_{j\in\NN_0} \in c_0(\NN_0)$ we have
	\begin{align}\label{eq:C^n}
		\hat{C_{\varphi}}^n(x) &:= (D \circ C_{\varphi} \circ D^{-1})^n(x) = (D \circ C_{\varphi}^n \circ D^{-1})(x) \nonumber \\[5pt]
		&= (D \circ C_{\varphi^n} \circ D^{-1})(x) = \left( \sum_{k=\varphi^n(j^-)+1}^{\varphi^n(j)} x_k \right)_{j\in\NN_0} \quad \text{ for each } n \in \NN.
	\end{align}
	Again, by abuse of notation we are setting $\varphi^n(0^-):=-1$ for every $n \in \NN$. Now, consider the dense linear subspaces $\Sigma_{00}(\NN_0) \subset c_{00}(\NN_0) \subset c_0(\NN_0)$ as defined after Remark~\ref{Rem:1-isom}, and note that for this particular case we have the equality
	\[
	\Sigma_{00}(\NN_0) = \left\{ x=(x_j)_{j\in\NN_0} \in c_{00}(\NN_0) \ ; \ \sum_{k=0}^{\infty} x_k = \sum_{k=0}^{|\supp(x)|} x_k = 0 \right\}.
	\]
	Let $B : c_{00}(\NN_0)\longrightarrow c_{00}(\NN_0)$ be the (linear) map acting on each $x=(x_j)_{j\in\NN_0} \in c_{00}(\NN_0)$ as
	\begin{equation}\label{eq:B}
		[B(x)]_k =
		\begin{cases}
			0 & \text{ if } k=0, \\[5pt]
			\tfrac{x_0}{\left|\varphi(0)\right|} & \text{ if } 1\leq k\leq \varphi(0), \\[5pt]
			\tfrac{x_j}{\dist(\varphi(j),\varphi(j^-))} & \text{ if } \varphi(j^-)+1\leq k\leq \varphi(j) \text{ with } j\in\NN.
		\end{cases}
	\end{equation}
	Note that the restriction $B : \Sigma_{00}(\NN_0)\longrightarrow \Sigma_{00}(\NN_0)$, which will still be denoted by $B$, is also well defined. We will check that $\hat{C_{\varphi}}$ satisfies the Frequent Hypercyclicity Criterion with respect to $B$:
	\begin{enumerate}[(I)]
		\item Consider any $x = (x_j)_{j\in\NN_0} \in \Sigma_{00}(\NN_0)$. Since $\varphi(0)>0$ and we are assuming (iv), it follows that $\lim_{n\to\infty} |\varphi^n(0)| = \infty$ so that we can choose $N_0 \in \NN$ fulfilling that $\varphi^{n}(0)\geq |\supp(x)|$ for all $n\geq N_0$. Thus, for each $n\geq N_0$ and by \eqref{eq:C^n} we have that
		\[
		\hat{C_{\varphi}}^n(x) = \left( \sum_{k=0}^{\varphi^{n}(0)} x_k, \sum_{k=\varphi^{n}(0)+1}^{\varphi^{n}(1)} x_k, \sum_{k=\varphi^{n}(1)+1}^{\varphi^{n}(2)} x_k, ... \right) = \left( \sum_{k=0}^{|\supp(x)|} x_k, 0, 0, ... \right) = (0,0,0,...),
		\]
		because $x \in \Sigma_{00}(\NN_0)$. It trivially follows that $\sum_{n=0}^{\infty} \hat{C_{\varphi}}^n(x)$ converges unconditionally.
		
		\item Now we claim that $\sum_{n=0}^{\infty} B^n(x)$ converges unconditionally for every $x \in \Sigma_{00}(\NN_0)$. To check it recall the notation after Remark~\ref{Rem:1-isom} from which $e^{(l)}=(\delta_{l,j})_{j\in\NN_0} \in c_0(\NN_0)$ for each $l \in \NN_0$. With this notation, since $B$ is a linear map originally defined on $c_{00}(\NN_0)$, since $\Sigma_{00}(\NN_0) \subset c_{00}(\NN_0)$, and since $B(e^{(0)})$ can be expressed as a linear combination of the vectors $e^{(1)}, e^{(2)}, ..., e^{(\varphi(0))}$, then the proof of (II) will be complete as soon as we check that:
		\begin{enumerate}[--]
			\item \textit{the series $\sum_{n=0}^{\infty} B^n(e^{(l)})$ converges unconditionally for every $l \in \NN$.}
		\end{enumerate}
		Thus, let us fix some $l \in \NN$. Given any $n \in \NN$, using \eqref{eq:B} inductively it follows that
		\begin{equation}\label{eq:supp}
			\supp\left(B^n(e^{(l)})\right) = \{ k \in \NN \ ; \ \varphi^n(l^-)+1\leq k\leq \varphi^n(l) \}.
		\end{equation}
		On the other hand, using assumption (iv) it is not hard to recursively construct a sequence of non-negative integers $(n_j)_{j\in\NN_0}$ with $n_0=0$ and such that
		\begin{equation}\label{eq:dist}
			\dist\left(\varphi^{n_j}(k),\varphi^{n_j}(k^-)\right) \geq 2 \quad \text{ for all } j \in \NN \text{ and } \varphi^{n_{j-1}}(l^-)+1\leq k\leq \varphi^{n_{j-1}}(l),
		\end{equation}
		where we are taking $\varphi^{n_0}=\varphi^{0}$ as the identity map in $\NN_0$, i.e.\ $\varphi^{n_0}(l^-)+1=l=\varphi^{n_0}(l)$. Hence, using \eqref{eq:B}, \eqref{eq:supp} and \eqref{eq:dist} one can easily prove by induction that
		\begin{equation}\label{eq:norm->0}
			\left\| B^n(e^{(l)}) \right\|_{\infty} = \max_{\varphi^n(l^-)+1\leq k\leq \varphi^n(l)} \left| [B^n(e^{(l)})]_k \right| \leq \frac{1}{2^j} \quad \text{ whenever } \sum_{i=0}^{j} n_i\leq n< \sum_{i=0}^{j+1} n_i.
		\end{equation}
		By \eqref{eq:supp} and the increasing condition of $\varphi$ we have that the supports of $B^n(e^{(l)})$ and $B^m(e^{(l)})$ do not intersect for $n\neq m$. Thus, since by \eqref{eq:norm->0} we have that $\|B^n(e^{(l)})\|_{\infty} \to 0$ as $n\to\infty$, it follows that the series $\sum_{n=0}^{\infty} B^n(e^{(l)})$ converges to some element in $c_0(\NN_0)$, and this convergence is unconditional because $(e^{(k)})_{k\in\NN_0}$ is an unconditional Schauder basis for $c_0(\NN_0)$.
		
		\item Combining the expressions \eqref{eq:C} and \eqref{eq:B} we have that $\hat{C_{\varphi}}B(x)=x$ for every $x \in \Sigma_{00}(\NN_0)$. 
	\end{enumerate}
	Since $C_{\varphi}$ and $\hat{C_{\varphi}}$ are topologically conjugate via an isometric isomorphism, it follows that $C_{\varphi}$ satisfies the Frequent Hypercyclicity Criterion and the proof of Theorem~\ref{The:NN_0} is now complete.
\end{proof}

\begin{remark}\label{Rem:NN_0}
	The statement and proof of Theorem~\ref{The:NN_0} admit several remarks:
	\begin{enumerate}[(a)]
		\item Note that Question~\ref{Q:first} is now partially solved, for the symbols that are strictly increasing self-maps of the set of non-negative integers, via Theorem~\ref{The:NN_0}.
		
		\item The condition ``\textit{$\varphi(0)>0$}'' in Theorem~\ref{The:NN_0} is necessary since, by \cite[Proposition~5.1]{ColonnaMar2025_MJM_composition}, the map~$\varphi$ can not have fixed points when $C_{\varphi}$ is hypercyclic. However, it seems very likely that an operator~$C_{\varphi}$ could be hypercyclic on $\Lc_0(\NN_0)$ even if the other condition ``\textit{$\varphi:\NN_0\longrightarrow\NN_0$ is strictly increasing}'' is not satisfied. Relaxing this last hypothesis could be a future matter of study.
		
		\item Under the assumptions ``\textit{$\varphi:\NN_0\longrightarrow\NN_0$ is strictly increasing}'' and ``\textit{$\varphi(0)>0$}'', then $C_{\varphi}$ is bounded in $\Lc_0(\NN_0)$ if and only if $\varphi$ is a Lipschitz map. Indeed, under these assumptions we have that
		\[
		\lim_{|j|\to\infty} |\varphi(j)| = \infty,
		\]
		and hence $C_{\varphi}$ is bounded in $\Lc_0(\NN_0)$ if and only if $\varphi$ is Lipschitz by \cite[Theorem~3.6]{ColonnaMar2025_MJM_composition}. 
		
		\item Using \eqref{eq:C} we can express the operator $\hat{C_{\varphi}}$ as an infinite-matrix acting on the coordinates of the sequence space $c_0(\NN_0)$. In fact, we can consider the infinite matrix $A_{\varphi} = (a_{i,j})_{i,j\in\NN_0}$ with
		\[
		A_{\varphi} =
		\begin{matrix}
			\text{\tiny $i=0$} \\ \text{\tiny $i=1$} \\ \text{\tiny $i=2$} \\[2.5pt] \vdots \\[2.5pt] \text{\tiny $i$} \\ \vdots
		\end{matrix}
		\begin{pmatrix}
			\overset{\text{\tiny $j=0$}}{1} & \overset{\text{\tiny $j=1$}}{1} & \cdots & \overset{\text{\tiny $\varphi(0)$}}{1} & \overset{\text{\tiny $\varphi(0)+1$}}{0} & \cdots & \overset{\text{\tiny $\varphi(1)$}}{0} & \overset{\text{\tiny $\varphi(1)+1$}}{0} & .... & \overset{\text{\tiny $\varphi(i^-)$}}{0} & \overset{\text{\tiny $\varphi(i^-)+1$}}{0} & \cdots \\
			0 & 0 & \cdots & 0 & 1 & \cdots & 1 & 0 & .... & 0 & 0 & \cdots \\
			0 & 0 & \cdots & 0 & 0 & \cdots & 0 & 1 & .... & 0 & 0 & \cdots \\
			\vdots & \vdots & \cdots & \vdots & \vdots & \cdots & \vdots & \vdots & .... & \vdots & \vdots & \cdots\\
			0 & 0 & \cdots & 0 & 0 & \cdots & \vdots & \vdots & .... & 0 & 1 & \cdots \\
			\vdots & \vdots & \cdots & \vdots & \vdots & \cdots & \vdots & \vdots & .... & \vdots & \vdots & \cdots
		\end{pmatrix}
		\]
		As far as we know this class of operators has not been considered or named in Linear Dynamics, and the author proposes to call them ``{\em accumulative backward shifts}''.
		
		\item The expression \eqref{eq:C} also induces bounded operators, which we still denote by $\hat{C_{\varphi}}$, on the classical spaces $\ell^p(\NN_0)$ for $1\leq p< \infty$. However, the respective operator $\hat{C_{\varphi}} : \ell^1(\NN_0) \longrightarrow \ell^1(\NN_0)$ cannot be hypercyclic since it is easily checked that $\|\hat{C_{\varphi}}\|_1 = 1$. For the other $\ell^p(\NN_0)$-spaces with $1<p<\infty$, the existence of $p$-summable sequences that are not $1$-summable may be used to prove that some of the respective operators $\hat{C_{\varphi}} : \ell^p(\NN_0) \longrightarrow \ell^p(\NN_0)$ could be hypercyclic. Characterizing the dynamics of the ``{\em $\ell^p(\NN_0)$-accumulative backward shifts}'' could be a future matter of study.
		
		\item Similar results to Theorem~\ref{The:NN_0} may be obtained for composition operators acting on rooted trees other than $T=\NN_0$ as described above, which could also be a future matter of study. However, the main objective of this paper is to provide some examples of frequently hypercyclic composition operators, and to expose the isometric isomorphism exhibited in Remarks~\ref{Rem:1-isom} as a tool to simplify the study of $\Lc(T)$, $\Lc_0(T)$ and of some operators acting there. Hence, for simplicity, here we will only consider rooted trees indexed by $\NN_0$ with $0$ as their root (see also Example~\ref{Exa:first} below).
	\end{enumerate}
\end{remark}

\begin{example}\label{Exa:first}
	Consider the composition operator described in \cite[Examples~4.6~and~6.7]{ColonnaMar2025_MJM_composition}: we let $T$ be the rooted tree indexed by $\NN_0$, again with $0 \in \NN_0$ as its root and with two vertex $k,l \in \NN_0$ being adjacent whenever $|k-l|=1$, and we let $\varphi:\NN_0\longrightarrow\NN_0$ with $\varphi(j)=2j+1$ for each $j\in\NN_0$. Hence:
	\begin{enumerate}[(a)]
		\item It was argued in \cite[Example~4.6]{ColonnaMar2025_MJM_composition} that the respective operator $C_{\varphi}:\Lc(\NN_0)\longrightarrow\Lc(\NN_0)$ is bounded with $\|C_{\varphi}\|_{\Lc}=2$, that its spectral radius is $r(C_{\varphi})=2$, and that its spectrum is $\sigma(C_{\varphi})=2\cl{\DD}$ because its point spectrum is $\sigma_p(C_{\varphi})=2\DD$, where we are letting $\DD:=\{ z \in \CC \ ; \ |z|<1 \}$. However, we must highlight that in \cite[Example~4.6]{ColonnaMar2025_MJM_composition}, the $C_{\varphi}$-eigenvector $f^{(\lambda)}:T\longrightarrow\CC$ considered for each $C_{\varphi}$-eigenvalue $\lambda \in 2\DD\setminus\{0\}$ was defined as $f^{(\lambda)}(j):=(j+1)^{\mu}$ for each $j \in \NN_0$ and some fixed $\mu \in \CC$ with $2^{\mu}=\lambda$. The reader is now invited to look at the (at least complex) calculations needed and done in \cite[Example~4.6]{ColonnaMar2025_MJM_composition} to ensure that the function $f^{(\lambda)}$ belongs to $\Lc_0(\NN_0)$.
		
		\item It was argued in \cite[Example~6.7]{ColonnaMar2025_MJM_composition} that the respective operator $C_{\varphi}:\Lc(\NN_0)\longrightarrow\Lc(\NN_0)$ is mixing, since it satisfies a nice criterion proved for general trees (see \cite[Theorem~6.6]{ColonnaMar2025_MJM_composition}). However, this result was not enough to imply further dynamical properties such as that of frequent hypercyclicity or that of Devaney chaos (i.e.\ hypercyclicity plus a dense set of periodic points).
	\end{enumerate}
	With the theory developed in this paper we can extremely simplify the (in our opinion technical) proof of the equality $\sigma_p(C_{\varphi})=2\DD$ stated in part (a) of this example, but also to strongly improve the result stated in part (b) via the Frequent Hypercyclicity Criterion. Indeed, let us briefly justify all the facts stated above via the results and remarks proved in this paper, based on the isometric isomorphism established in Remark~\ref{Rem:1-isom} and the topological conjugacy established in Remark~\ref{Rem:2-conjugacy}:
	\begin{enumerate}[(a')]
		\item Note that~$\varphi$ is strictly increasing, that $\varphi(0)=1>0$, and that $\varphi$ is a Lipschitz map with constant~$L_{\varphi}=2$. Hence, by Remark~\ref{Rem:NN_0} we know that $C_{\varphi}$ is bounded in $\Lc_0(\NN_0)$ and $\varphi$ satisfies the initial assumptions of Theorem~\ref{The:NN_0}. In particular, by Corollary~\ref{Cor:norm} we know that $\|C_{\varphi}\|=2$ and, since $\varphi^n(j)=(j+1)\cdot 2^n-1$ for all $n \in \NN$ and $j \in \NN_0$, by \eqref{eq:C^n} we have that
		\begin{equation}\label{eq:C^n2}
			\hat{C_{\varphi}}^n(x) = \left( \sum_{k=j\cdot 2^n}^{(j+1)\cdot 2^n-1} x_k \right)_{j\in\NN_0} \quad \text{ for each } x=(x_j)_{j\in\NN_0} \in c_0(\NN_0) \text{ and each } n \in \NN.
		\end{equation}
		Again by Corollary~\ref{Cor:norm} it follows that $\|C_{\varphi}^n\|_{\Lc} = \|\hat{C_{\varphi}}^n\|_{\infty}=2^n$ and hence $r(C_{\varphi})=r(\hat{C_{\varphi}}) = 2$. In addition, given any $\lambda \in 2\DD\setminus\{0\}$ we can consider the sequence $x^{(\lambda)} = (x^{(\lambda)}_j)_{j\in\NN_0}$ with
		\[
		x^{(\lambda)}_0=1 \quad \text{ and } \quad x^{(\lambda)}_{2^j} = x^{(\lambda)}_{2^j+1} = ... = x^{(\lambda)}_{2^{j+1}-1} = \left(\tfrac{\lambda}{2}\right)^j \cdot (\lambda-1) \quad \text{ for each } j\geq 0. 
		\]
		Since $|\lambda|<2$ it is obvious that $x^{(\lambda)} \in c_0(\NN_0)$, and the equality $\hat{C_{\varphi}}(x^{(\lambda)})=\lambda x^{(\lambda)}$ follows from a direct application of \eqref{eq:C^n2} with $n=1$. The reader may now compare these simple arguments, based on the existence of the isometrically isomorphic topological conjugacy exhibited in Remark~\ref{Rem:2-conjugacy}, with the complex calculations done in \cite[Example~4.6]{ColonnaMar2025_MJM_composition}.
		
		\item Note that $\dist(\varphi(j),\varphi(j-1))=2$ for every $j \in \NN$. Hence, by statement (v) of Theorem~\ref{The:NN_0}, the composition operator $C_{\varphi}:\Lc(\NN_0)\longrightarrow\Lc(\NN_0)$ satisfies the Frequent Hypercyclicity Criterion. In particular it is frequently hypercyclic, Devaney chaotic and mixing.
	\end{enumerate}
	The infinite-matrix form of the operator $\hat{C_{\varphi}}$ considered in Example~\ref{Exa:first} when it acts on $c_0(\NN_0)$ is
	\[
	A_{\varphi} =
	\begin{matrix}
		\text{\tiny $i=0$} \\ \text{\tiny $i=1$} \\ \text{\tiny $i=2$} \\[1pt] \vdots \\[1pt] \text{\tiny $i$} \\ \vdots
	\end{matrix}
	\begin{pmatrix}
		\overset{\text{\tiny $j=0$}}{1} & \overset{\text{\tiny $j=1$}}{1} & \overset{\text{\tiny $j=2$}}{0} & \overset{\text{\tiny $j=3$}}{0} & \overset{\text{\tiny $j=4$}}{0} & \overset{\text{\tiny $j=5$}}{0} & 0 & 0 & \overset{\text{\tiny $2i$}}{0} & \overset{\text{\tiny $2i+1$}}{0} & 0 & \cdots \\
		0 & 0 & 1 & 1 & 0 & 0 & 0 & 0 & 0 & 0 & 0 & \cdots \\
		0 & 0 & 0 & 0 & 1 & 1 & 0 & 0 & 0 & 0 & 0 & \cdots \\
		\vdots & \vdots & \vdots & \vdots & \vdots & \vdots & \ddots & \ddots & \vdots & \vdots & \vdots & \cdots\\
		0 & 0 & 0 & 0 & 0 & 0 & 0 & 0 & 1 & 1 & 0 & \cdots \\
		\vdots & \vdots & \vdots & \vdots & \vdots & \vdots & \vdots & \vdots & \vdots & \vdots & \ddots & \ddots
	\end{pmatrix},
	\]
	with two consecutive $1$'s in the columns $j=2i$ and $j=2i+1$ of the $i$-th row for $i \in \NN_0$.
\end{example}

\section{Further comments and future lines of research}\label{Sec_4:conclusions}

The main objective of this paper was to exhibit simple examples of frequently hypercyclic composition operators acting on the little Lipschitz space of a rooted tree. To do so, we established in Section~\ref{Sec_2:preliminar} an isometric isomorphism between the Lipschitz-type spaces under consideration and the corresponding sequence spaces $\ell^{\infty}$ and $c_0$. We then used this identification to obtain alternative expressions for the operators studied. In the present section, we further explore the scope of this isometric representation and we also formulate several open problems about the operators we considered in this paper.

We proceed by examining additional applications of the isometry established in Section~\ref{Sec_2:preliminar}. Actually, we recall that for a directed tree $T$ with a root $\zero \in T$ as defined in Section~\ref{Sec_2:preliminar}: in~\cite{CoEas2010}, the spaces considered in this paper $\Lc(T)$ and $\Lc_0(T)$ were introduced for the first time; in~\cite{ACoEas2013}, it was defined the so-called {\em weighted Lipschitz space} $\Lc_{\sbf{w}}(T)$ and the respective {\em little space} $\Lc_{\sbf{w},0}(T)$; and in~\cite{ACoEas2012}, a generalization called the {\em iterated logarithmic} (and the respective {\em little}) {\em Lipschitz spaces} appeared. If now we let $X$ be any of these Lipschitz spaces, and $X_0$ the associated little space, in each case there exists a concrete sequence $c = (c_{v})_{v\in T}$ of positive real numbers with $c_{\zero}=1$ for which
\[
X = \left\{ f:T\longrightarrow\CC \ ; \ \sup_{v \in T^*} \ c_v|f(v)-f(v^-)| < \infty \right\},
\]
and
\[
X_0 = \left\{ f:T\longrightarrow\CC \ ; \ \lim_{|v|\to\infty} \ c_v|f(v)-f(v^-)| = 0 \right\},
\]
both of them endowed with the norm
\[
\|f\| := |f(\zero)| + \sup_{v \in T^*} \ c_v|f(v)-f(v^-)| \quad \text{ for each } f \in X.
\]
It is not hard to check that, if one considers the equivalent norm
\[
\|f\|_X := \max\left\{ |f(\zero)| , \sup_{v \in T^*} \ c_v|f(v)-f(v^-)| \right\} \quad \text{ for each } f \in X,
\]
then the maps $D:(X,\|\cdot\|_X)\longrightarrow(\ell^{\infty}(T,c),\|\cdot\|_{\infty,c})$ and $D:(X_0,\|\cdot\|_X)\longrightarrow(c_0(T,c),\|\cdot\|_{\infty,c})$ defined as in Remark~\ref{Rem:1-isom} are again isometric isomorphisms, where
\[
\ell^{\infty}(T,c) := \left\{ x=(x_v)_{v\in V} \in \CC^{T} \ ; \ \sup_{v \in T} \ c_v|x_v| < \infty \right\}
\]
and
\[
c_0(T,c) := \left\{ x=(x_v)_{v\in V} \in \CC^{T} \ ; \ \lim_{|v|\to\infty} \ c_v|x_v| = 0 \right\}
\]
are the classical weighted $\ell^{\infty}$ and $c_0$ spaces as defined in \cite[Section~4.1]{GrPe2011_book}, \cite{GrPa2023} and \cite{GrPa2023_arXiv}, both of them endowed with the {\em weighted supremum norm}
\[
\|x\|_{\infty,c} := \sup_{v \in T} \ c_v|x_v| \quad \text{ for each } x=(x_v)_{v\in V} \in \ell^{\infty}(T,c).
\]
In addition, these sequence spaces are isometrically isomorphic to $(\ell^{\infty}(T),\|\cdot\|_{\infty})$ and $(c_0(T),\|\cdot\|_{\infty})$ as defined in Remark~\ref{Rem:1-isom}. See \cite[pages 96 and 353]{GrPe2011_book} for further details.

As a consequence of these observations, although these Lipschitz-type spaces over rooted trees were originally motivated by the well-known and extensively studied \textit{Bloch space} (formed by holomorphic functions whose derivative grows no faster than the density of the Poincar\'e metric; see \cite[Section~1]{CoEas2010}), the isometric identification established above shows that they can be naturally realised as concrete representations of the classical sequence spaces $\ell^{\infty}$ and $c_0$. From this perspective, several results obtained in \cite{ACoEas2011,ACoEas2012,ACoEas2013,ACoEas2014,CoEas2010,CoEas2012,CoMar2017,ColonnaMar2025_MJM_composition,MarRi2020} admit alternative proofs, and in some cases more direct arguments, by means of the correspondence described after Remark~\ref{Rem:1-isom}. In particular, the identities
\[
D(\chi_{S_w}) = e^{(w)}, \quad 
D^{-1}(e^{(w)}) = \chi_{S_w}, \quad 
D(\chi_{\{w\}}) = \dx^{(w)}, \quad 
D^{-1}(\dx^{(w)}) = \chi_{\{w\}},
\]
for each $w \in T$, allow one to transfer well-known structural properties of the classical spaces $\ell^{\infty}$~and~$c_0$ or about their duals to the Lipschitz-type spaces under consideration.

On the other hand, the study of operators on these Lipschitz-type spaces presents several natural features of interest. First, such operators are typically defined through simple symbols (such as $\psi$ for multiplication or $\varphi$ for composition operators, see \cite{CoEas2010,CoMar2017,ColonnaMar2025_MJM_composition}), which makes their formulation particularly transparent. Secondly, via the map $D$, these operators admit equivalent representations as operators acting on the classical sequence spaces $\ell^{\infty}$ and $c_0$, as illustrated in Sections~\ref{Sec_2:preliminar} and~\ref{Sec_3:existence}. This correspondence makes it possible to analyse and classify them in terms of shift, multiplication and composition operators on sequence spaces. The availability of different, yet equivalent, representations of the same operator often provides additional structural insight, as the examples and arguments developed in Sections~\ref{Sec_2:preliminar} and~\ref{Sec_3:existence} demonstrate. Thus, it appears natural to continue the systematic study of the operators considered so far on these Lipschitz-type spaces over rooted trees, and to extend the analysis to related classes, such as weighted composition or weighted shift operators.

From the perspective of Linear Dynamics, many natural problems can be posed. For the rest of the paper we leave some concrete open questions regarding the composition operators $C_{\varphi}$ considered in this paper, but also for their complex multiples $\lambda C_{\varphi}$ with $\lambda \in \CC$. We start asking about the possibility of distinguishing some of the main dynamical properties considered in the literature:

\begin{problem}\label{P:mix_chaos_freq}
	Can the composition operators (or their complex multiples) acting on little Lipschitz spaces be used to distinguish the notions of mixing, Devaney chaos and frequent hypercyclicity?
\end{problem}

Since the little Lipschitz spaces are isomorphic to certain $c_0$-spaces as argued above, Problem~\ref{P:mix_chaos_freq} seems a relevant question because the first operator that distinguished frequent hypercyclicity from chaos was a very specific backward shift acting on $c_0(\NN_0)$. This backward shift is frequently hypercyclic, but neither mixing nor chaotic (see \cite[Section~5]{BaGri2007} and then \cite{GriMaMe2021_book} for examples of the converse).

\begin{remark}
	We include in this remark what we known about Problem~\ref{P:mix_chaos_freq}:
	\begin{enumerate}[(a)]
		\item The hypercyclic composition operators considered in this paper fulfill the Frequent Hypercyclicity Criterion, so that they can not be used to solve Problem~\ref{P:mix_chaos_freq}.
		
		\item About the composition operators considered in \cite{ColonnaMar2025_MJM_composition} and the two criteria given there:
		\begin{enumerate}[--]
			\item The proof of \cite[Theorem~6.1]{ColonnaMar2025_MJM_composition}, combined with that of \cite[Proposition~6.2]{ColonnaMar2025_MJM_composition}, can be strengthened to show that the symbols $\varphi$ fulfilling the hypothesis of \cite[Theorem~6.1]{ColonnaMar2025_MJM_composition} are such that the operator $\lambda C_{\varphi} : \Lc_0(T)\longrightarrow\Lc_0(T)$ satisfies the Frequent Hypercyclicity Criterion for all $|\lambda|>1$. In particular, these operators $\lambda C_{\varphi}$ can not be used to solve Problem~\ref{P:mix_chaos_freq}.
			
			\item The hypothesis in \cite[Theorem~6.6]{ColonnaMar2025_MJM_composition} seem less demanding that those of \cite[Theorem~6.1]{ColonnaMar2025_MJM_composition}, and it is not completely clear if satisfying such result implies or not satisfying the Frequent Hypercyclicity Criterion. Clarifying this fact could be the starting point to solve Problem~\ref{P:mix_chaos_freq}. 
		\end{enumerate}
	\end{enumerate}
	In addition, the notion of (usual) hypercyclicity for composition operators and their multiples has not been distinguished from that of mixing. It could be interesting to start by distinguishing them. One should probably consider trees more sophisticated than $\NN_0$ to address these problems. 
\end{remark}

As a second question, and in view of \cite[Proposition~6.2,~Corollary~6.3~and~Theorem~6.6]{ColonnaMar2025_MJM_composition}, we would like to propose the following:

\begin{problem}\label{P:finite_preimage}
	Let $C_{\varphi}$ be a composition operator on a little Lipschitz space $\Lc_0(T)$ and assume that~$\lambda C_{\varphi}$ is hypercyclic for some $|\lambda|\geq 1$. Is $\{ n \in \NN \ ; \ \varphi^{-n}(\{v\})\neq\emptyset \}$ a finite set for every vertex $v \in T$?
\end{problem}

By \cite[Proposition~5.3]{ColonnaMar2025_MJM_composition} we have that every symbol $\varphi$ fulfilling that $\lambda C_{\varphi}$ is hypercyclic, for some $\lambda \in \CC$, is an injective map. Note that, given any vertex $v \in T$ and any positive integer $n \in \NN$, we only have two possibilities for the set $\varphi^{-n}(\{v\}):=\{ w \in T \ ; \ \varphi^n(w)=v \}$ when $\varphi$ is injective:
\begin{enumerate}[--]
	\item we either have $\varphi^{-n}(\{v\})=\{u\}$ for some vertex $u \in T$;
	
	\item or we have that $\varphi^{-n}(\{v\})$ is empty.
\end{enumerate}
Thus, when $\varphi$ is injective, it follows that ``\textit{the set $\{ n \in \NN \ ; \ \varphi^{-n}(\{v\})\neq\emptyset \}$ is finite for every $v \in T$}'' if and only if ``\textit{the set $\bigcup_{n\in\NN} \varphi^{-n}(\{v\})$ is finite for every $v \in T$}''. Conversely, but also when $\varphi$ is injective, given any vertex $v \in T$ we have that
\[
\text{\textit{the set $\bigcup_{n\in\NN} \varphi^{-n}(\{v\})$ is infinite}} \quad \text{ if and only if } \quad \text{\textit{$\{ n \in \NN \ ; \ \varphi^{-n}(\{v\})\neq\emptyset \}=\NN$}}.
\]
Our conjecture is that Problem~\ref{P:finite_preimage} has a positive answer. In fact, if $\{ n \in \NN \ ; \ \varphi^{-n}(\{v\})\neq\emptyset \} = \NN$ for a certain vertex $v \in T$, it seems that all the $\lambda C_{\varphi}$-orbits of the functions in $\Lc_0(T)$ that are near enough to the characteristic function $\chi_{\{v\}} \in \Lc_0(T)$ would hardly approach the zero-vector of the underlying space for $|\lambda|\geq 1$. However, we can not discard the possibility of constructing a sufficiently complex map~$\varphi$, acting on a sufficiently complex tree~$T$, to answer Problem~\ref{P:finite_preimage} in the negative.

Although the mentioned condition ``\textit{the set $\{ n \in \NN \ ; \ \varphi^{-n}(\{v\})\neq\emptyset \}$ is finite for every $v \in T$}'' was relaxed in \cite[Theorem~6.6]{ColonnaMar2025_MJM_composition}, all the hypercyclic operators of the type $\lambda C_{\varphi}:\Lc_0(T)\longrightarrow\Lc_0(T)$ provided by the authors in \cite{ColonnaMar2025_MJM_composition} fulfilled such condition. Also, every map $\varphi:\NN_0\longrightarrow\NN_0$ considered here in Section~\ref{Sec_3:existence} fulfills that the set $R_j := \{ n \in \NN \ ; \ \varphi^{-n}(\{j\})\neq\emptyset \}$ is finite for every $j \in \NN_0$. Indeed, since~$\varphi$ was assumed to be strictly increasing with $\varphi(0)>0$, it follows that the cardinality of the set~$R_j$ is at most $j$. In view of these comments, we consider worth it to provide an example of composition operator $C_{\varphi}$ whose symbol $\varphi$ fulfills that the set $\{ n \in \NN \ ; \ \varphi^{-n}(\{v\})\neq\emptyset \}=\NN$ for some $v \in T$ but such that $\lambda C_{\varphi}$ is hypercyclic for some $\lambda \in \CC$, with $|\lambda|<1$ in view of Problem~\ref{P:finite_preimage}:

\begin{example}
	Consider the rooted tree $T$ indexed by $\ZZ$, with $0 \in \NN_0$ as its root, and where two vertices $k,l \in \ZZ$ are adjacent if and only if $|k-l|=1$. Consider now the map
	\[
	\varphi(j) :=
	\begin{cases}
		j+1, & \text{ if } j<0, \\[5pt]
		2j+1, & \text{ if } j\geq 0.
	\end{cases}
	\]
	Note that $\varphi$ is Lipschitz with constant $L_{\varphi}=2$ and that $\lim_{|j|\to\infty} |\varphi(j)| = \infty$, which shows that $C_{\varphi}$ is bounded in $\Lc_0(\ZZ)$ by \cite[Theorem~3.1]{ColonnaMar2025_MJM_composition}. Moreover, we have the equality
	\[
	\bigcup_{n\in\NN} \varphi^{-n}(\{0\}) = \{ -n  \ ; \ n \in \NN \} \quad \text{ which shows that } \quad \{ n \in \NN \ ; \ \varphi^{-n}(\{0\})\neq\emptyset \} = \NN.
	\]
	This is an infinite set, contrary to what happened for the symbols considered in Section~\ref{Sec_3:existence}. Note also that $\{ n \in \NN \ ; \ \varphi^{-n}(\{j\})\neq\emptyset \} = \NN$ for each negative integer $j \in \ZZ\setminus\NN_0$ and for the positive integers of the form $j=2^n-1$ with $n \in \NN$. One can check that $\varphi$ fulfills the assumptions of \cite[Theorem~6.6]{ColonnaMar2025_MJM_composition} for every $\lambda \in \CC$ with $\tfrac{1}{2}<|\lambda|<1$ (consider $c=1$ in \cite[Theorem~6.6,~Condition~(2)]{ColonnaMar2025_MJM_composition}). It follows that the operator $\lambda C_{\varphi} : \Lc_0(\ZZ)\longrightarrow\Lc_0(\ZZ)$ is mixing whenever $\tfrac{1}{2}<|\lambda|<1$.
\end{example}

For the next and last questions we use the {\em essential spectrum}. It is known that a weakly-mixing bounded linear operator $A:X\longrightarrow X$ on a separable Banach space $X$ admits a {\em hypercyclic subspace} (i.e.\ a closed infinite-dimensional subspace formed by hypercyclic vectors but the zero-vector) if and only if its essential spectrum $\sigma_e(A)$ intersects the closed unit disk $\cl{\DD}$ (see \cite{Lopez2024_IMRN}). In particular, if $\ker(A)$ is infinite-dimensional, then $0 \in \sigma_e(A)\cap\cl{\DD}$ (see \cite[Corollary~7.1]{Lopez2024_IMRN}). For our composition operators it is not hard to check that if $T\setminus\varphi(T)$ is an infinite set, then $\ker(C_{\varphi})$ is infinite-dimensional. Indeed, note that $\{ \chi_{\{v\}} \ ; \ v \in T\setminus\varphi(T) \}$ is a linearly independent set contained in $\ker(C_{\varphi}) \subset \Lc_0(T)$. Thus, we can ask the following natural questions:

\begin{problem}\label{P:infinite}
	Let $C_{\varphi}$ be a composition operator on a little Lipschitz space $\Lc_0(T)$ and assume that~$\lambda C_{\varphi}$ is hypercyclic for some $\lambda \in \CC$ with $|\lambda|\leq 1$. Does it follow that $T\setminus\varphi(T)$ is an infinite set?
\end{problem}

\begin{problem}\label{P:subspace}
	Let $C_{\varphi}$ be a composition operator on a little Lipschitz space $\Lc_0(T)$ and assume that~$\lambda C_{\varphi}$ is hypercyclic for some $\lambda \in \CC$ with $|\lambda|\leq 1$. Does $\lambda C_{\varphi}$ admit a hypercyclic subspace?
\end{problem}

In Problems~\ref{P:infinite} and Problem~\ref{P:subspace} we have avoided the option $|\lambda|>1$ because the composition operator considered in Example~\ref{Exa:NN_0} fulfills that:
\begin{enumerate}[--]
	\item The map $\varphi : \NN_0\longrightarrow\NN_0$ is such that $\NN_0\setminus \varphi(\NN_0) = \{0\}$, i.e. $T\setminus\varphi(T)$ is finite in this case.
	
	\item The operator $\lambda C_{\varphi} : \Lc_0(\NN_0)\longrightarrow\Lc_0(\NN_0)$ can be checked to be hypercyclic and even mixing for every $\lambda \in \CC$ with $|\lambda|>1$ by applying \cite[Theorem~6.1]{ColonnaMar2025_MJM_composition}.
	
	\item The operator $\lambda C_{\varphi}$ does not admit a hypercyclic subspace when $|\lambda|>1$. Indeed, since $\hat{C_{\varphi}} = P_0 + B$ as argued in Example~\ref{Exa:NN_0}, we have that $\lambda C_{\varphi}$ is topologically conjugate via an isometric isomorphism to the compact perturbation $\lambda P_0 + \lambda B$ of the Rolewicz operator $\lambda B$. It is well-known that such operator does not admit a hypercyclic subspace (see \cite[Remark~5.9~and~Corollary~7.2]{Lopez2024_IMRN}).
\end{enumerate}
Our conjecture is that Problem~\ref{P:infinite} should have a positive answer, and this could probably be achieved using that the set $\{ |\lambda|^n \cdot \dist(\varphi^n(v),\varphi^n(v^-)) \ ; \ n \in \NN \}$ is unbounded whenever the operator $\lambda C_{\varphi}$ is hypercyclic by \cite[Proposition~5.4]{ColonnaMar2025_MJM_composition}. However, note that a positive answer to Problem~\ref{P:infinite} does not directly imply a positive answer to Problem~\ref{P:subspace}, unless one shows that every hypercyclic operator of the type $\lambda C_{\varphi}$ is weakly-mixing (see \cite[Proposition~7.9]{Lopez2024_IMRN} for more details).

To conclude, we would like to mention other possible lines of research apart from that related to hypercyclicity. In fact, recall that in \cite{CoMar2017} the authors showed that the multiplication operators acting on little Lipschitz spaces (but also their adjoints) can not be hypercyclic. Hence, we propose to study for these operators (and also for the shifts considered in \cite{MarRi2020}) other dynamical notions that have recently appeared in Linear Dynamics such as ({\em mean}) {\em Li-Yorke} and {\em distributional chaos} (see \cite{BerBoMarPe2011,BerBoPe2020}), {\em shadowing}, {\em hyperbolicity} and {\em structural stability} (see \cite{BerCiDarMePu2018,BerMe2021,CiGoPu2021}) or {\em recurrence} (see \cite{BoGrLoPe2022_JFA,GriLoPe2025_AMP,LoMe2025_JMAA}). The isomorphism and conjugacies established in this paper seem good tools to look into these topics.

\section*{Funding}

The author of this paper was partially supported by MCIN/AEI/10.13039/501100011033/FEDER, UE, Project PID2022-139449NB-I00.

\section*{Acknowledgments}

The author would like to thank the anonymous reviewer for the careful comments received, which have significantly improved the paper.

{\footnotesize
$\ $\\

\textsc{Antoni L\'opez-Mart\'inez}: Universitat Polit\`ecnica de Val\`encia, Institut Universitari de Matem\`atica Pura i Aplicada, Edifici 8E, 4a planta, 46022 Val\`encia, Spain. e-mail: alopezmartinez@mat.upv.es

}

\end{document}